\documentclass{article}

\usepackage{amsmath,amssymb,amsthm}
\usepackage{bm}
\usepackage{mathrsfs}
\usepackage[all]{xy}
\usepackage{booktabs}
\usepackage{fullpage}
\usepackage{hyperref}
\usepackage{mathtools}
\usepackage[]{todonotes}

\usepackage[abbrev,msc-links]{amsrefs}

\hypersetup{
 hidelinks,
 bookmarksopen=true
}

\newcommand{\rl}{\mathbb{R}}
\newcommand{\cx}{\mathbb{C}}

\newcommand{\CB}{\mathcal{B}}

\newcommand{\CO}{\mathcal{O}}

\newcommand{\ai}{\sqrt{-1}}

\newcommand{\inj}{\hookrightarrow}

\newcommand{\cst}{\mathrm{const}}

\newcommand{\tr}{\mathrm{tr}}
\newcommand{\actson}{\curvearrowright}
\newcommand{\prj}{\mathbb{P}}

\theoremstyle{plain}
\newtheorem{theorem}{Theorem}[section]
\newtheorem{lemma}[theorem]{Lemma}

\theoremstyle{definition}
\newtheorem{definition}[theorem]{Definition}

\newtheorem{example}[theorem]{Example}
\theoremstyle{definition}
\newtheorem{remark}[theorem]{Remark}
\newtheorem{remarks}[theorem]{Remarks}
\newtheorem{problem}[theorem]{Problem}

\begin{document}

\title{Expected centre of mass of the random Kodaira embedding}
\author{Yoshinori Hashimoto}

\maketitle

\begin{abstract}
Let $X \subset \mathbb{P}^{N-1}$ be a smooth projective variety. To each $g \in SL (N , \mathbb{C})$ which induces the embedding $g \cdot X \subset \mathbb{P}^{N-1}$ given by the ambient linear action we can associate a matrix $\bar{\mu}_X (g)$ called the centre of mass, which depends nonlinearly on $g$. With respect to the probability measure on $SL (N , \mathbb{C})$ induced by the Haar measure and the Gaussian unitary ensemble, we prove that the expectation of the centre of mass is a constant multiple of the identity matrix for any smooth projective variety.
\end{abstract}

\section{Introduction and the statement of the main result}

Let $X$ be a complex smooth projective variety, and $\iota : X \inj \prj (H^0(X,L)^{\vee}) \cong \prj^{N-1}$ be the Kodaira embedding defined with respect to a very ample line bundle $L$ on $X$, where $N := \dim H^0 (X,L)$. There is a natural $SL(N , \cx)$-action on the Kodaira embedding $\iota \mapsto g \cdot \iota$ given by the ambient linear action $SL (N , \cx) \actson \prj^{N-1}$. For each $g \in SL(N , \cx)$ we can define an $N \times N$ hermitian matrix $\bar{\mu}_X (g)$, called the centre of mass of the embedding $g \cdot \iota : X \inj \prj (H^0(X,L)^{\vee})$ (see \S \ref{scmmcom} for more details). This plays an important role in K\"ahler geometry, and depends on $g \in SL(N , \cx)$ in a highly nonlinear manner. For example, when the automorphism group of $(X,L)$ is discrete, there exists $g \in SL(N , \cx)$ such that $\bar{\mu}_X (g)$ is a constant multiple of the identity matrix if and only if the embedding $\iota : X \inj  \prj^{N-1}$ is Chow stable \cite{Luo,Zhang}, which is an important yet subtle algebro-geometric property of $X \subset \prj^{N-1}$.

The following seems to be a natural question to ask.

\begin{problem} \label{precom}
	Let $d \sigma$ be a probability measure on $SL(N , \cx)$. Compute the expectation
	\begin{equation*}
		\mathbb{E}[\bar{\mu}_X (g)] = \int_{g \in SL(N , \cx)} \bar{\mu}_X (g) d \sigma .
	\end{equation*}
\end{problem}

In spite of its apparent simplicity, this is a nontrivial problem since $\bar{\mu}_X (g)$ depends nonlinearly on $g$. The main result of this paper is the following.

\begin{theorem} \label{mainthm}
	Let $X \subset \prj^{N-1}$ be a smooth projective variety. With respect to the probability measure on $SL (N ,\cx)$ defined by the Haar measure on $SU(N)$ and an absolutely continuous unitarily invariant measure of finite volume on $\CB:= SL(N, \cx) / SU(N)$ via the natural fibration structure, $\mathbb{E}[\bar{\mu}_X (g)]$ is a constant multiple of the identity matrix.
\end{theorem}

See \S\ref{scrdmt} for the details of the measure on $SL (N ,\cx)$ as stated in the above, defined by the fibration $SU(N) \to SL(N , \cx) \to \CB$; it is also discussed therein that the measure on $SL (N , \cx)$ induced by the Gaussian unitary ensemble on $\CB$ (Example \ref{exgue}) satisfies all the properties stated in the theorem. We also note that the absolute continuity of the measure on $\mathcal{B}$ is meant to be with respect to the Haar measure on $\mathcal{B}$.

The study of K\"ahler and Fubini--Study metrics in connection to the probability theory, such as the random matrix theory, has been an active area of research. There are works e.g.~\cite{Ber11,Ber14,Ber17,Ber18,Ber20} by Berman, and \cite{FKZ11,FKZ12a,FKZ13,FlZel,Kle14,KZ14,KZ16,SongZel12} by Ferrari, Flurin, Klevtsov, Song, Zelditch. On the other hand, probabilistic aspects of the centre of mass $\bar{\mu}_X (g)$ does not seem to have been actively investigated in the aforementioned works, which is the focus of the present paper.

As pointed out in the above, whether $\bar{\mu}_X (g)$ itself is a constant multiple of the identity matrix depends on the Chow stability of $X \subset \prj^{N-1}$ by the result of Luo \cite{Luo} and Zhang \cite{Zhang}. Such subtleties disappear, however, when we take the average over $g \in SL(N , \cx)$ as in Theorem \ref{mainthm}.

While the main point of Theorem \ref{mainthm} is that $\mathbb{E}[\bar{\mu}_X (g)]$ is a constant multiple of the identity for any smooth projective variety, it implies in particular that the expectation $\mathbb{E}[\bar{\mu}_X (g)]$ keeps being a constant multiple of the identity for the embedding $X \inj \prj (H^0 (X , L^{\otimes k})^{\vee})$ for any higher exponent $k \gg 1$. This may be interesting in the study of the large $N$ behaviour of random K\"ahler metrics, initiated by Ferrari--Klevtsov--Zelditch \cite{FKZ13}. One may hope, for example, that $\mathbb{E}[\bar{\mu}_X (g)]$ keeps being a multiple of the identity for $k \gg 1$ gives a nontrivial constraint to the large $N$ asymptotic behaviour of their theory.

We also note that we can prove the following unitary version of Theorem \ref{mainthm}, although the proof (given in \S \ref{scpfmntm2}) is much easier.

\begin{theorem} \label{mainthm2}
	Let $X \subset \prj^{N-1}$ be a smooth projective variety. With respect to the Haar measure $d \sigma_{SU}$ on $SU(N)$, the expectation
	\begin{equation*}
		\mathbb{E}_{SU}[\bar{\mu}_X (u)] := \int_{u \in SU(N)} \bar{\mu}_X (u) d \sigma_{SU}
	\end{equation*}
	of the centre of mass of $X  \subset \prj^{N-1}$is a constant multiple of the identity matrix.
\end{theorem}

We can also define a variant $\bar{\mu}_{X , \nu}$ of the centre of mass, as in Definition \ref{defcomnu}, by fixing a volume form $d \nu$ on $X$. It turns out that Theorems \ref{mainthm} and \ref{mainthm2} easily extend to this variant, as explained in Remarks \ref{rmthm2var} and \ref{rmthmvar}, essentially because $\bar{\mu}_{X , \nu} (g)$ depends on $g \in SL(N, \cx)$ in a much less nonlinear manner than $\bar{\mu}_{X} (g)$. The author is grateful to the anonymous referee for suggesting this point to him.

\begin{remark}
	Although we shall only treat $SL(N , \cx)$ and $SU(N)$ throughout this paper, the determinant one condition does not play any significant role. We can run exactly the same argument for $GL (N , \cx)$ and $U(N)$ to get the same results, in fact with a slightly simpler proof.
\end{remark}

\subsection*{Acknowledgements}

The author is immensely grateful to Jean-Yves Welschinger for generously sharing with him a key insight that was very important in this work. He is also grateful to Hiroyuki Fuji, Yoshiyuki Kagei, Tatsuya Miura, and Yoshihiro Tonegawa for helpful discussions. He thanks the anonymous referees for helpful comments, particularly for suggesting to extend the results to the variant $\bar{\mu}_{X , \nu}$ in Definition \ref{defcomnu}.

This version of the article has been accepted for publication, after peer review (when applicable) but is not the Version of Record and does not reflect post-acceptance improvements, or any corrections. The Version of Record is available online at: http://dx.doi.org/10.1007/s12220-021-00778-y

\section{Preliminaries}

\subsection{Random matrices} \label{scrdmt}

Our aim is to define a class of probability measures on $SL(N , \cx)$ which has some good properties as in the statement of Theorem \ref{mainthm}. The precise description of such measures is given in Definition \ref{dfhguems}, but that needs to be accompanied by a review of some elementary results in the theory of random matrices; the details can be found e.g.~in \cite{agz,deift1,deift2,mehta} or any other standard textbooks on random matrices.

Let $\CB := SL(N, \cx) / SU(N)$ be the left coset space, which can be naturally identified with the set of all positive definite hermitian matrices (of determinant one) on $\cx^N$, which gives $SL(N , \cx)$ a natural structure of a principal $SU(N)$-bundle
\begin{displaymath}
		\xymatrix{ SU(N) \ar[r] & SL(N , \cx)  \ar[d]^{\pi} \\  & \CB  } %\xymatrixcolsep{5pc}
\end{displaymath}
by the projection
\begin{equation} \label{eprbdpi}
	\pi : SL(N, \cx) \ni g \mapsto g g^{*}  \in \CB ,
\end{equation}
where $g^{*}$ stands for the hermitian conjugate of $g$ with respect to the hermitian form represented by the identity matrix on $\cx^N$. Throughout, we shall write $e$ for the identity in $SL(N , \cx)$ or $SU(N)$.
\begin{definition} \label{dfhguems}
We set our notational convention, and the definition of the measure $d \sigma$ on $SL(N , \cx)$, as follows.
\begin{itemize}
	\item We write $d \sigma_{SU}$ for the Haar measure on $SU(N)$ of unit volume.
	\item We fix a measure $d \sigma_B$ on $\CB$, and assume that $d \sigma_B$ is absolutely continuous, unitarily invariant, and of finite volume.
	\item Given a measure $d \sigma_B$ on $\CB$ and $d \sigma_{SU}$ on $SU(N)$, the measure defined on $SL (N , \cx)$ via the fibration structure (\ref{eprbdpi}) is denoted by $d \sigma$.
\end{itemize}
\end{definition}

Given any measure $d \sigma$ on $SL(N , \cx)$ as defined above, it is immediate that $d \sigma$ is of finite volume (see also Lemma \ref{lmvldbca}). Henceforth without loss of generality we shall assume
\begin{equation} \label{equvlsg}
	\int_{SL(N , \cx)} d \sigma = 1
\end{equation}
by scaling, i.e.~$d \sigma$ is a probability measure on $SL(N , \cx)$.

We have a more explicit formula for $d \sigma$, which follows immediately from the above definition.
\begin{lemma} \label{lmvldbca}
	Suppose that $d \sigma$ is a probability measure on $SL(N , \cx)$ defined as in Definition \ref{dfhguems}. If $\phi : SL(N , \cx) \to \rl$ is a bounded measurable function, we have
	\begin{equation*}
		\int_{SL(N , \cx)} \phi (g) d \sigma (g) = \frac{1}{\mathrm{Vol} (\CB )} \int_{\CB}  d \sigma_B (hh^*)\int_{\pi^{-1} (hh^*)}  \phi (h u) d \sigma_{SU} (u) ,
	\end{equation*}
	where $\mathrm{Vol} (\CB ) := \int_{\CB} d \sigma_B$ is the volume of $\CB$ with respect to $d \sigma_B$, and $h \in SL(N , \cx)$ is a hermitian matrix such that $\pi (g) = hh^*$.
\end{lemma}

We now recall some basic facts on the Euclidean volume form (or its associated Lebesgue measure) on the $N \times N$ hermitian matrices (not necessarily positive definite or of determinant one), induced by the natural Euclidean metric. By unitarily diagonalising a hermitian matrix $\tilde{H}$ as $\tilde{H} = u^{-1} \Lambda u$ for $u \in U(N)$ and $\Lambda = \mathrm{diag} (\lambda_1 , \dots , \lambda_N ) \in \rl^N$, we can write (see e.g.~\cite[\S 2]{fyodorov}, \cite[Chapter 5]{deift1}, \cite[Chapter 2]{deift2})
\begin{equation*}
	d \tilde{H} = \Delta^2 ( \lambda ) \; \prod_{i=1}^N d \lambda_i \; d \sigma_{U},
\end{equation*}
where $d \sigma_U$ is the Haar measure on $U(N)$ and $\Delta^2 ( \lambda )$ is the square of the Vandermonde determinant
\begin{equation*}
	\Delta ( \lambda ) := \prod_{1 \le i \neq j \le N} ( \lambda_i - \lambda_j ).
\end{equation*}
We consider the volume form on $\CB$, which consists of positive definite hermitian matrices $H$ of determinant one, induced by the Euclidean metric as above. Setting $\lambda_N = \prod_{i=1}^{N-1} \lambda_i^{-1}$ and carrying out the computation exactly as in \cite[\S 2]{fyodorov}, we find
\begin{equation} \label{eqdhpdcp}
	dH  = \Delta^2 ( \lambda ) \gamma (\lambda ) \; \prod_{i=1}^{N-1} d \lambda_i \; d \sigma_{SU},
\end{equation}
for some smooth positive function $\gamma \in C^{\infty} (\rl^{N-1}_{>0} , \rl_{>0})$ on the $(N-1)$-fold direct product of positive real numbers $\rl^{N-1}_{>0}$. The notation $\delta ( \log \det \tilde{H} ) d \tilde{H}$ using the delta function is also used e.g.~in \cite[\S 4.1]{FKZ13} to denote $d H $ as in (\ref{eqdhpdcp}).

Now returning to our original setting, we note that the measure $d \sigma_B$ on $\CB$ being absolutely continuous  means that we can write
\begin{equation} \label{eqdstrho}
	d \sigma_B = \rho (H) dH ,
\end{equation}
where $dH$ is as defined in (\ref{eqdhpdcp}) and $\rho : \CB \to [ 0 , + \infty ) $ is a measurable function (called the Radon--Nikodym density) which is known to exist by the Radon--Nikodym theorem. Moreover, $d \sigma_B$ being of finite volume implies
\begin{equation} \label{eqrhfv}
	\int_{\CB} \rho (H) dH < + \infty .
\end{equation}

Finally, $d \sigma_B$ being unitarily invariant means that $d \sigma_B (H) = d \sigma_B (u H u^{-1})$ for all $H \in \CB$ and $u \in SU(N)$, which is equivalent to saying that $\rho (H)$ depends only on the eigenvalues $\lambda_1 , \dots , \lambda_N$ of $H$ (where $\lambda_N = \prod_{i=1}^{N-1} \lambda_i^{-1}$). By abuse of notation we also write $\rho (\lambda_1 , \dots , \lambda_{N-1} )$ for $\rho (H)$. With this notation, the finite volume condition (\ref{eqrhfv}) translates to
\begin{equation} \label{eqrhfvpd}
	\int_{\rl^{N-1}_{>0}} \rho (\lambda_1 , \dots , \lambda_{N-1} ) \Delta^2 ( \lambda ) \gamma (\lambda) \; \prod_{i=1}^{N-1} d \lambda_i < + \infty .
\end{equation}

\begin{example} \label{exgue}
	An example of the measure as defined in Definition \ref{dfhguems} can be given by the \textbf{Gaussian unitary ensemble} on $\CB$ (or more precisely, the Gaussian unitary ensemble restricted to the set of positive definite hermitian forms $\CB$) defined by the following Radon--Nikodym density
\begin{equation*}
	\rho(H) = \exp \left(- \frac{1}{2} \tr(H^2) \right) .
\end{equation*}
Recalling (\ref{eqdhpdcp}), the Gaussian unitary ensemble $d \sigma_B$ can be written more explicitly as 
\begin{equation*}
		d \sigma_{B } = \cst . \Delta^2 (\lambda ) \gamma (\lambda) \exp \left( - \frac{1}{2} \sum_{i=1}^N \lambda_i^2 \right)  \; \prod_{i=1}^{N-1} d \lambda_i \; d \sigma_{SU},
\end{equation*}
with $\lambda_N = \prod_{i=1}^{N-1} \lambda_i^{-1}$, up to an overall positive constant. With the Haar measure $d \sigma_{SU}$ on the fibres of $\pi$, the Gaussian unitary ensemble defines a probability measure $d \sigma$ on $SL(N , \cx)$ satisfying all the properties of Definition \ref{dfhguems}.
\end{example}

\begin{remark}
	A well-known theorem \cite[Chapter 2]{mehta} in fact shows that, if $\rho(H)$ is absolutely continuous, unitarily invariant, and moreover the diagonal entries and the real and imaginary parts of the off-diagonal entries of $H$ are statistically independent, $\rho (H)$ must be of the form $\exp (- (a \mathrm{tr} (H^2) + b \mathrm{tr} (H) +c) )$ for some constants $a>0$, $b, c \in \rl$. %\cite[p97]{deift1}
\end{remark}

\begin{example}
	Yet another example of the measure $d \sigma_B$ on $\CB$ is given by the \textbf{heat kernel measure}, which is defined by the heat kernel on the homogeneous manifold $\CB = SL (N , \cx ) / SU(N )$. More explicitly, the heat kernel measure $d \sigma_{B, t}$, defined for each $t >0$, can be written in terms of the Lebesgue measure $dH$ on $\CB$ and the eigenvalues $\lambda'_1 , \dots , \lambda'_N$ of $\log H$ (i.e.~the $\rl^N$-part of the polar coordinates on $\CB$) as
	\begin{equation*}
		d \sigma_{B , t} := \cst . \frac{\Delta ( \lambda' )}{\Delta ( e^{\lambda'})} \exp \left( - \frac{1}{4t} \sum_{i=1}^N (\lambda'_i)^2 \right) dH,
	\end{equation*}
	up to an overall positive constant. The above measure satisfies all the properties in Definition \ref{dfhguems} for each $t >0$. See \cite[Proposition 3.2]{Gangolli68} and \cite[\S 3.1]{KZ16} for more details.
\end{example}

\begin{remark}
	Klevtsov--Zelditch \cite[\S 5]{KZ14} considered the measure $\exp (- \gamma S_{\nu } (H)) dH$,  where $\gamma >0 $ is a constant and $S_{\nu }$ is a certain functional defined on $\CB$ with respect to a volume form $\nu$ on $X$, for the study of the partition function of some field theory. Interesting as it is, the unitary invariance $S_{\nu } (H) = S_{\nu } (u H u^{-1})$ (for all $u \in SU(N)$) does not seem to hold for $S_{\nu}$, so Theorem \ref{mainthm} does not seem to apply to the case when we use $\exp (- \gamma S_{\nu } (H)) dH$ as a measure on $\CB$.
\end{remark}

\begin{remark} \label{rmrfbs}
	Note that the measure $d \sigma_B$ or $d \sigma$ as discussed in the above depends on the fixed hermitian form on $\cx^N$, represented by the identity matrix. This corresponds to the choice of the reference basis $\{ Z_i \}_{i=1}^{N}$ that we take to identify $H^0 (X,L)$ with $\cx^N$ in \S \ref{scmmcom}.
\end{remark}

\subsection{Moment maps and the centre of mass} \label{scmmcom}

We review the ingredients from complex geometry that we need in this paper. Let $X$ be a complex smooth projective variety of complex dimension $n$, with a very ample line bundle $L$ and the associated embedding $\iota : X \inj \prj (H^0 (X , L)^{\vee})$.

We fix a basis for $H^0(X,L)$ once and for all and identify $\prj (H^0(X,L)^{\vee}) \cong \prj^{N-1}$, where $N := \dim H^0(X,L)$; we also note that the basis we fixed here can be identified with an orthonormal basis for the hermitian form represented by the identity matrix on $\cx^N \cong H^0(X,L)$ (see also Remark \ref{rmrfbs}). With respect to such a reference basis, we write $[Z_1 : \cdots : Z_N]$ for the homogeneous coordinates for $\prj^{N-1}$. Furthermore, by abuse of terminology, we also write $\{ Z_i \}_{i=1}^N$ for the reference basis itself. Pick $g \in SL(N , \cx)$ and write
\begin{equation} \label{eqzijg}
	Z_i(g) := \sum_{j=1}^N g_{ij} Z_j ,
\end{equation}
where $g_{ij}$ is the matrix representation of $g$ with respect to the basis $\{ Z_i \}_{i=1}^N$. Note that $\{ Z_i(g)\}_{i=1}^N$ defines a new basis for $H^0(X,L)$. Throughout, we shall write
\begin{equation*}
	H_g := (g^{-1})^{*} g^{-1} = (g g^{*})^{-1}
\end{equation*}
for the the positive definite hermitian matrix on $H^0(X,L)$ that has $\{ Z_i(g)\}_{i=1}^N$ as its orthonormal basis. The hermitian conjugate (with respect to the basis $\{ Z_i \}_{i=1}^N$) will be denoted by $*$, and the special unitary group $SU(N)$ is always meant to preserve the hermitian form $H_e$ which has $\{ Z_i \}_{i=1}^N$ as its orthonormal basis.

For each positive definite hermitian form on $\cx^N$, it is a foundational result in complex geometry that we have a K\"ahler metric on $\prj^{N-1}$ called the Fubini--Study metric (see e.g.~\cite[Chapter 0, \S 2]{GrHar} for more details).

\begin{definition} \label{dffsmc}
	The \textbf{Fubini--Study metric} $\tilde{\omega}_{H_e}$ on $\prj^{N-1}$ defined by $H_e$ is an $SU(N)$-invariant K\"ahler metric on $\prj^{N-1}$, whose explicit formula on $\cx^{N-1} = \{ Z_1 \neq 0 \} \subset \prj^{N-1}$ is given by 
	\begin{equation*}
		\tilde{\omega}_{H_e} = \frac{\ai}{2 \pi} \partial \bar{\partial} \log \left( 1 + \sum_{i=2}^{N} |z_i|^2 \right)
	\end{equation*}
	where $z_i := Z_i / Z_1$ for $i=2 , \dots , N$. By abuse of terminology, the restriction of $\tilde{\omega}_{H_e}$ to $\iota (X) \subset \prj^{N-1}$ is also called the Fubini--Study metric on $\iota (X)$, and written $\omega_{H_e} := \iota^* \tilde{\omega}_{H_e}$.
\end{definition}

While the above definition is often stated for a fixed hermitian matrix, different hermitian matrices lead to different Fubini--Study metrics; for the hermitian matrix $H_g$, the associated Fubini--Study metric $\tilde{\omega}_{H_g}$ can be written, on $\cx^{N-1} = \{ Z_1(g) \neq 0 \} \subset \prj^{N-1}$, as
\begin{equation*}
	\tilde{\omega}_{H_g} = \frac{\ai}{2 \pi} \partial \bar{\partial} \log \left( 1 + \sum_{i=2}^{N} |z_i (g)|^2 \right)
\end{equation*}
by replacing $z_i$ with $z_i (g) := Z_i (g) / Z_1 (g)$. While the isometry group of $\tilde{\omega}_{H_g} $ is isomorphic to $SU(N)$, it is not the same $SU(N)$ that we fixed above; while the $SU(N)$ as above preserves the hermitian form $H_e$, in general it does not preserve $H_g$ if $g \neq e$. Recall also that $\omega_{H_g} := \iota^* \tilde{\omega}_{H_g} \in c_1 (L)$ for all $g \in SL (N , \cx )$.

From the above definition, by writing in terms of polar coordinates $z_i (g) = r_i (g) e^{\ai \theta_i (g)}$ we have
\begin{equation} \label{eqvffspn}
	\tilde{\omega}_{H_g}^{N-1} = \frac{1}{\left( 1 + \sum_{i=2}^{N-1} r_i (g)^2 \right)^{N-1}} \prod_{i=2}^{N} \frac{r_i (g) d r_i (g) \wedge d \theta_i (g)}{2 \pi}.
\end{equation}
Note also that the restriction of $\tilde{\omega}^n_{H_g}$ to $\iota (X)$ defines a volume form on $\iota (X)$, which we write as
\begin{equation*}
	d \nu_{H_g} := \frac{\omega_{H_g}^n}{n!}. %( \tilde{\omega}_{H_g} |_{\iota (X)} )^n .
\end{equation*}
The total volume of $X$ with respect to $d \nu_{H_g}$ can be computed as
\begin{equation} \label{eqvlc1l}
	\int_X d \nu_{H_g} = \int_X c_1 (L)^n / n! =: \mathrm{Vol} (X,L),
\end{equation}
which depends only on $(X,L)$ and is independent of $g \in SL(N, \cx )$.

Recall that ($\ai $ times) the moment map $\mu_{SU} : \prj^{N-1} \to \ai \mathfrak{su} (N)$ for the $SU(N)$-action on $\prj^{N-1}$ is given by
\begin{equation*}
	\mu_{SU} ([x_1 : \cdots : x_N ] )_{ij} = \frac{x_i \bar{x}_j}{\sum_{l=1}^N |x_l|^2} - \frac{\delta_{ij}}{N} ,
\end{equation*}
where $\delta_{ij}$ is the Kronecker delta and the subscript $ij$ stands for the $(i,j)$-th entry of the $N \times N$ matrix. The second term $\delta_{ij}/N$ is just to make $\mu_{SU}$ trace-free. Observing that $SU(N)$ acts transitively on $\prj^{N-1}$, we find that $\mu_{SU}$ naturally defines a map $\mu_{SU,p} : SU(N) \to \ai \mathfrak{su} (N)$ by $\mu_{SU,p} (u) := \mu_{SU} (u p)$ where $p \in \prj^{N-1}$ is a fixed reference point.

We now consider the ``complexified'' version of the above moment map, defined for $SL (N , \cx) = SU(N)^{\cx}$. We fix a reference point $p \in \prj^{N-1}$ represented by the homogeneous coordinates $[Z_1 : \cdots : Z_N ]$, and observe that for each $g \in SL (N , \cx)$ the point $g p \in \prj^{N-1}$ is represented by $[Z_1 (g) : \cdots : Z_N (g) ]$ in terms of the notation (\ref{eqzijg}). We then define an $N \times N$ hermitian matrix $\mu_{p} (g) \in \ai \mathfrak{u} (N)$ whose $(i,j)$-th entry is given by
\begin{equation} \label{eqdfmup}
	\mu_{p} (g)_{ij} = \frac{Z_i(g) \overline{Z_j(g)}}{\sum_{l=1}^N |Z_l(g)|^2} .
\end{equation}
This corresponds to the first term of $\mu_{SU}$ at the point $g p$; note that $g p$ is in the $SU(N)^{\cx}$-orbit of $p$. We choose not to normalise the trace of $\mu_p (g)$ to be zero, to be consistent with the notation in the literature. The centre of mass, which plays an important role in this paper, is defined for $g \in SL(N , \cx)$ and the embedded variety $\iota : X \inj \prj^{N-1}$ as the integral
\begin{equation} \label{eqdfmupcm}
	\bar{\mu}_X (g) := \int_{p \in \iota(X)} \mu_{p} (g) d \nu_{H_g} .
\end{equation}
We summarise the above in the following formal definition.

\begin{definition} \label{dfcomass}
	The \textbf{centre of mass} $\bar{\mu}_X (g) $, defined for $g \in SL(N , \cx)$ and $\iota : X \inj \prj^{N-1}$, is a hermitian matrix of size $N$ whose $(i,j)$-th entry is given in terms of the notation (\ref{eqzijg}) by
	\begin{equation*}
		\bar{\mu}_X (g)_{ij} := \int_{\iota(X)} \frac{Z_i(g) \overline{Z_j(g)}}{\sum_{l=1}^N |Z_l(g)|^2} d \nu_{H_g}, 
	\end{equation*}
	where $d \nu_{H_g}$ is the measure on $\iota (X)$ defined by the Fubini--Study metric on $\prj^{N-1}$ with respect to $H_g$, and integrates with respect to the variables $\{ Z_i \}_{i=1}^N$ over the locus $\{ [Z_1 : \cdots : Z_N] \in \iota (X) \} \subset \prj^{N-1}$.
\end{definition}

It is easy to see how $\mu_p (g)$ in (\ref{eqdfmup}) changes when $g$ is pre-multiplied by a unitary matrix $u$, as in the following lemma.

\begin{lemma} \label{eqcmutf}
For any $g \in SL(N, \cx)$, $u \in SU(N)$, and $p \in \prj^{N-1}$, we have
	\begin{equation*} 
		\mu_p (ug) = u \cdot \mu_p (g) \cdot  u^{*}
	\end{equation*}
\end{lemma}

\begin{proof}
	It is an obvious consequence of $\sum_{l=1}^N |Z_l(g)|^2 = \sum_{l=1}^N |Z_l(ug)|^2$ for any unitary matrix $u$. \qed
\end{proof}
	
Note, on the other hand, that we do not have an analogous formula for $\mu_p (gu)$.

\begin{remarks} \label{rmcmbd}
We observe some other elementary properties of the centre of mass which immediately follow from the definition.
	\begin{enumerate}
	\item Both $\mu_{p} (g)$ and $\bar{\mu}_X (g)$ are positive definite as a hermitian matrix for each $g \in SL (N , \cx)$.
	\item We observe that $\bar{\mu}_X (g)$ is nothing but the integral of $\mu_{p} (e)$ over $p \in g \cdot \iota(X)$ with respect to $d \nu_{H_g}$; $\bar{\mu}_X (g)$ can be regarded as the centre of mass of the Kodaira embedding $g \cdot \iota (X) \subset \prj^{N-1}$.
	\item $\bar{\mu}_X (g)$ is independent of the overall scaling of $g$, so depends only on its class in $PSL (N , \cx)$. Moreover, we observe that each entry of the integrand $\mu_{p} (g)$ of the centre of mass is manifestly bounded as a function of $g \in SL(N , \cx)$ for each $p \in \prj^{N-1}$.
	\end{enumerate}
\end{remarks}

Computing the centre of mass is in general difficult since $\bar{\mu}_X (g)$ depends on $g \in SL(N , \cx)$ (and the embedding $\iota : X \inj \prj^{N-1}$) in a highly nonlinear manner and the size $N$ of the matrices is typically large. However, there are some special cases in which we can explicitly compute it.

\begin{example} \label{expnlebd}
Take $X:= \prj^{N-1}$ and $L := \CO_{\prj^{N-1}}(1)$. Then, by using (\ref{eqvffspn}) and the polar coordinates for $\cx^{N-1}$, we find that $\bar{\mu}_{\prj^{N-1}} (g)$ is a constant multiple of the identity matrix for all $g \in SL(N, \cx)$; this computation is well-known to the experts and reduces to the periodicity of the angle coordinates, but the details can be found e.g.~in \cite[Lemma 2.7]{yhhilb}. In particular, $\mathbb{E}[ \bar{\mu}_{\prj^{N-1}} (g) ]$ is a constant multiple of the identity matrix for any probability measure $d \sigma$ on $SL (N , \cx)$.
\end{example}

\begin{example} \label{expnlebdtp}
	The above method using the polar coordinates also work for the case when $\prj^{n}$ is embedded in a higher dimensional projective space by the Veronese embedding, i.e.~when $L = \CO_{\prj^n} (m)$ for $m>1$, and $\{ Z_i (g) \}_{i=1}^N$ is given by the monomial basis for $H^0 (\prj^n, \CO_{\prj^n} (m))$, where $N = \dim_{\cx} H^0 (\prj^n, \CO_{\prj^n} (m)) $. As in the previous example, $\bar{\mu}_{\prj^n} (g)$ can be easily seen to be a diagonal matrix for $g \in SL (N , \cx)$ such that $\{ Z_i (g) \}_{i=1}^N$ is a monomial basis. By appropriately scaling the monomial basis, we find that there exists $g \in SL (N , \cx)$ such that $\bar{\mu}_{\prj^n} (g)$ is a constant multiple of the identity, and the explicit scaling can be written down as in \cite[Example2.4]{ArLoi04}.
\end{example}

We also have a variant of the centre of mass, introduced by Donaldson \cite[\S 2]{DonNum} as follows.

\begin{definition} \label{defcomnu}
	Let $d \nu$ be a fixed volume form on $\iota (X)$. We define a variant $\bar{\mu}_{X , \nu} (g)$ of (\ref{eqdfmupcm}) by the following formula
	\begin{equation*}
	\bar{\mu}_{X , \nu} (g) := \int_{p \in \iota(X)} \mu_{p} (g) d \nu ,
	\end{equation*}
	in which we replaced $d \nu_{H_g}$ in (\ref{eqdfmupcm}) by the fixed volume form $d \nu$.
\end{definition}

As we shall see later, it is straightforward to extend the results for $\bar{\mu}_X (g)$ to the variant $\bar{\mu}_{X , \nu} (g)$; indeed, the volume form $d \nu$ not depending on $g$ means that $\bar{\mu}_{X , \nu} (g)$ depends on $g$ in a much less nonlinear manner than $\bar{\mu}_{X} (g)$, and the proof turns out to be simpler.

\subsection{Proof of Theorem \ref{mainthm2}} \label{scpfmntm2}
The properties of the centre of mass presented in \S \ref{scmmcom} are sufficient for the proof of Theorem \ref{mainthm2}, which is elementary. We compute
\begin{equation*}
	\mathbb{E}_{SU}[\bar{\mu}_X (u)] := \int_{u \in SU(N)} \bar{\mu}_X (u) d \sigma_{SU} = \int_{u \in SU(N)}  d \sigma_{SU} \int_{p \in \iota(X)} \mu_{p} (u) d \nu_{H_u}.
\end{equation*}
Note first that $d \nu_{H_u} = d \nu_{H_e}$ for all $u \in SU(N)$ since $H_u = (uu^*)^{-1} = H_e$. Lemma \ref{eqcmutf} further implies that the above is equal to
\begin{equation*}
	\mathbb{E}_{SU}[\bar{\mu}_X (u)] = \int_{u \in SU(N)}  d \sigma_{SU} \left( u \cdot \int_{p \in \iota(X)} \mu_{p} (e) d \nu_{H_e} \cdot u^* \right) .
\end{equation*}
We pick and fix an arbitrary $\eta \in SU(N)$, and observe that the group invariance of the Haar measure implies
\begin{align*}
	&\int_{u \in SU(N)}  d \sigma_{SU}(u) \left( u \cdot \int_{p \in \iota(X)} \mu_{p} (e) d \nu_{H_e} \cdot u^* \right) \\
	&= \int_{\eta u \in SU(N)}  d \sigma_{SU}(\eta u ) \left( \eta u \cdot \int_{p \in \iota(X)} \mu_{p} (e) d \nu_{H_e} \cdot u^* \eta^* \right) \\
	&= \int_{u \in SU(N)}  d \sigma_{SU}(u ) \left( \eta u \cdot \int_{p \in \iota(X)} \mu_{p} (e) d \nu_{H_e} \cdot u^* \eta^* \right) \\
	&= \eta \cdot \int_{u \in SU(N)}  d \sigma_{SU}(u) \left( u \cdot \int_{p \in \iota(X)} \mu_{p} (e) d \nu_{H_e} \cdot u^* \right) \cdot \eta^*,
\end{align*}
which implies that we have
\begin{equation*}
	\mathbb{E}_{SU}[\bar{\mu}_X (u)] = \eta \cdot \mathbb{E}_{SU}[\bar{\mu}_X (u)] \cdot \eta^*
\end{equation*}
for any $\eta \in SU(N)$. Recalling that the centre of mass $\bar{\mu}_X (u)$ is an $N \times N$ hermitian matrix, this implies that $\mathbb{E}_{SU}[\bar{\mu}_X (u)]$ must be a constant multiple of the identity matrix since it is a hermitian matrix that commutes with all elements of $SU(N)$. Noting that $\mathrm{tr} ( \bar{\mu}_X (u)) = \mathrm{Vol} (X,L)$ for all $u \in SU(N)$, we find more explicitly that
\begin{equation*}
	\mathbb{E}_{SU}[\bar{\mu}_X (u)] = \frac{\mathrm{Vol} (X , L )}{N} \cdot \mathrm{id}_{N \times N},
\end{equation*}
which completes the proof of Theorem \ref{mainthm2}.

\begin{remark} \label{rmthm2var}
	Note that the above proof applies word by word to prove
\begin{equation*}
	\mathbb{E}_{SU}[\bar{\mu}_{X ,\nu} (u)] = \frac{\mathrm{Vol} (X , L )}{N} \cdot \mathrm{id}_{N \times N}
\end{equation*}
for the variant in Definition \ref{defcomnu}, by noting that $d \nu$ is fixed and remains invariant under the $SU(N)$-action.
\end{remark}

\section{Proof of Theorem \ref{mainthm}}%\section{Proof of the main result}

Observe first that the definition of the centre of mass (\ref{eqdfmupcm}) implies
\begin{align*}
	\mathbb{E} [ \bar{\mu}_X (g) ] &=\int_{SL (N , \cx)} d \sigma(g) \int_{x \in \iota(X)} \mu_x (g) d \nu_{H_g} \\
	&=\int_{SL (N , \cx)} d \sigma(g) \int_{x \in \iota(X)} \mu_x (g) \frac{\omega^n_{H_g}}{\omega^n_{H_e}} d \nu_{H_e} ,
\end{align*}
where $\mu_x$ is as defined in (\ref{eqdfmup}) and we endow $SL(N , \cx ) \times \iota (X)$ with the product measure $d \sigma \times d \nu_{H_e} $. We swap the order of the above integrals by Fubini's theorem to find
\begin{equation*}
	\mathbb{E} [ \bar{\mu}_X (g) ] = \int_{x \in \iota(X)} d \nu_{H_e}  \int_{SL (N , \cx)}  \mu_x (g) \frac{\omega^n_{H_g} (x)}{\omega^n_{H_e} (x)}  d \sigma(g) .
\end{equation*}
%where $\left. \tilde{\omega}^n_{H_g} / \tilde{\omega}^n_{H_e}  \right|_x $ means the pullback of 

%\left. \frac{\tilde{\omega}^n_{H_g}}{\tilde{\omega}^n_{H_e}}  \right|_x

%which simplifies, by Proposition \ref{lmjac} and (\ref{eqfbcf}), to
%\begin{equation} \label{eqavcomacf}
%	\mathbb{E} [ \bar{\mu}_X (g) ] =  \int_{x\in \iota(X)} d \nu_{H_e} \int_{SL(N , \cx )} \mu_x (g )  d \sigma (g) ,
%\end{equation}
%since $d \sigma (g) = \varphi (g) d \tilde{\sigma} (g)$. 

We first fix $x \in \iota(X)$, pick a hermitian $h \in SL(N , \cx)$ such that $\pi (g) = hh^*$, and compute the second integral in the above as
\begin{align*}
	&\int_{SL(N , \cx)} \mu_x (g ) \frac{\omega^n_{H_g} (x)}{\omega^n_{H_e} (x)}  d \sigma (g) \\
	&= \frac{1}{\mathrm{Vol} (\CB )} \int_{\CB } d \sigma_{B} (hh^*) \int_{SU(N)} \mu_x (hu ) \frac{\omega^n_{H_{hu}} (x)}{\omega^n_{H_e} (x)} d \sigma_{SU} (u)
\end{align*}
by using Lemma \ref{lmvldbca}, where we note that each entry of $\mu_x (g)$ is bounded (Remark \ref{rmcmbd}) and that $\pi^{-1} (hh^*) = h \cdot SU(N)$. Observe that we may write $h = \eta \Lambda \eta^{*}$ for some $\eta \in SU(N)$ and a diagonal matrix $\Lambda = \mathrm{diag} (\Lambda_1 , \dots , \Lambda_N )$ which we can identify with a vector in $\rl^N$. With this notation we may write
\begin{equation*}
	\pi (g) = (hu) \cdot (hu)^{*} = hh^{*} = \eta \Lambda^2 \eta^{*} ,
\end{equation*}
where $u \in SU(N)$. We also note
\begin{equation*}
\omega_{H_g}=\omega_{H_{hu}} = \iota^* \left( \ai \partial \bar{\partial} \log \left( \sum_{i=1}^N \Lambda^2_i |Z_i(\eta^{-1} u)|^2 \right) \right),
\end{equation*}
which implies that $\omega^n_{H_g} (x) / \omega^n_{H_e} (x)$ is bounded over $SL(N , \cx)$, since an overall scaling of $\Lambda$ leaves the above metric invariant.

%with $\tilde{\omega}_{H_g} = \ai \partial \bar{\partial} \log \left( \sum_{i=1}^N \Lambda^2_i |Z_i(u)|^2 \right)$, when we write $g = \eta \Lambda u$ as before. We can then repeat the proof of Theorem \ref{mainthm} in this section by replacing (\ref{eqcmmmpij}) by
%\begin{equation*}
%	\int_{\prj^{N-1}} \frac{\Lambda_i \Lambda_j Z_i \bar{Z}_j}{\sum_{l=1} \Lambda_l^2 |Z_l|^2} \frac{\tilde{\omega}^n_{H_g}}{\tilde{\omega}^n_{H_e}} \tilde{\omega}^{N-1}_{H_e},
%\end{equation*}
%and observe that the subsequent computation applies word by word, since the only term that contains the angle coordinates is $Z_i \bar{Z}_j$ in the numerator and $\tilde{\omega}_{H_g}$ does not depend on $\eta$. While this clearly simplifies the proof of Theorem \ref{mainthm}, we leave the proof using the coarea formula as it is, since Proposition \ref{lmjac} could be of independent interest.

Thus, by writing $\tilde{\Lambda} := \Lambda^2$, the above integral may be written as
\begin{align*}
	\frac{1}{\mathrm{Vol} (\CB )} &\int_{\tilde{\Lambda} \in \rl^{N-1}_{>0} } \Delta^2 (\tilde{\Lambda} ) \gamma (\tilde{\Lambda} ) \rho (\tilde{\Lambda}) d \tilde{\Lambda} \\
	&\int_{SU(N)} d \sigma_{SU} (\eta ) \int_{SU(N)} \mu_x (\eta \Lambda \eta^{-1} u ) \Psi_x (\Lambda ,  \eta^{-1} u ) d \sigma_{SU} (u) ,
\end{align*}
by (\ref{eqdhpdcp}) and (\ref{eqdstrho}), where we set
\begin{equation*}
	\Psi_x (\Lambda ,  \eta^{-1} u ) := \frac{\omega^n_{H_{hu}} (x)}{\omega^n_{H_e} (x)} =\frac{ \iota^* \left( \ai \partial \bar{\partial} \log \left( \sum_{i=1}^N \Lambda^2_i |Z_i(\eta^{-1} u)|^2 \right) \right)^n (x) }{\omega^n_{H_e} (x)},
\end{equation*}
$\rho (\tilde{\Lambda})$ is the Radon--Nikodym density of $d \sigma_B$, and $\Delta $, $\gamma$ are as in (\ref{eqdhpdcp}). By the group invariance of the Haar measure, we have
\begin{align*}
&\int_{u \in SU(N)} \mu_x (\eta \Lambda \eta^{-1} u ) \Psi_x (\Lambda ,  \eta^{-1} u ) d \sigma_{SU} (u) \\
&= \int_{\eta u \in SU(N)} \mu_x (\eta \Lambda  u ) \Psi_x (\Lambda ,  u ) d \sigma_{SU} ( \eta u)  \\
&=\int_{u \in SU(N)} \mu_x (\eta \Lambda  u ) \Psi_x (\Lambda ,  u ) d \sigma_{SU} ( u) \\
&= \eta \left( \int_{u \in SU(N)}  \mu_x (\Lambda  u )  \Psi_x (\Lambda ,  u ) d \sigma_{SU} (u) \right) \eta^{*}
\end{align*}
by recalling Lemma \ref{eqcmutf} and noting that
\begin{equation} \label{defpsi}
	\Psi_x (\Lambda ,  u )= \frac{ \iota^* \left( \ai \partial \bar{\partial} \log \left( \sum_{i=1}^N \Lambda^2_i |Z_i( u)|^2 \right) \right)^n (x) }{\omega^n_{H_e} (x)} 
\end{equation}
does not depend on $\eta$, where the homogeneous coordinates $[Z_1 : \cdots : Z_N]$ are evaluated at $x \in \iota(X)$.

We are thus reduced to first computing
\begin{equation} \label{eqcmmmp}
	\int_{SU(N)} \mu_x (\Lambda u ) \Psi_x (\Lambda ,  u ) d \sigma_{SU} (u).
\end{equation}
We claim that the off-diagonal entries of the above integral are zero. Since any $x \in \iota(X) \subset \prj^{N-1}$ can be moved to $p_0=[1: 0: \cdots : 0]$ by the $SU(N)$-action, for the moment we assume without loss of generality that $x=p_0$, by using the $SU(N)$-invariance of the Haar measure. Since $p_0$ is fixed by the subgroup $S(U(1) \times U(N-1))$ of $SU(N)$, the integral (\ref{eqcmmmp}) is in fact an integral over $\prj^{N-1} = SU(N)/ S(U(1) \times U(N-1))$. We now recall that a group invariant measure on a homogeneous space (if exists) is unique up to an overall positive multiplicative constant by \cite[Chapter III, \S 4, Theorem 1]{nachbin}, which is a result credited to Weil in \cite{nachbin}. Thus, the measure on $\prj^{N-1}$ induced by the Haar measure $d \sigma_{SU}$ agrees, up to an overall constant multiple, with the $SU(N)$-invariant Fubini--Study measure $\tilde{\omega}^{N-1}_{H_e}$. Thus, by using the homogeneous coordinate system $[Z_1 : \dots : Z_N ]$ given by the reference basis, we find that the $(i,j)$-th entry of (\ref{eqcmmmp}) is equal to
\begin{equation} \label{eqcmmmpij}
	\int_{\prj^{N-1}} \frac{\Lambda_i \Lambda_j Z_i \bar{Z}_j}{\sum_{l=1} \Lambda_l^2 |Z_l|^2} \Psi_x (\Lambda ,  [Z_1 : \dots : Z_N ] ) \tilde{\omega}^{N-1}_{H_e}
\end{equation}
up to an overall constant multiple, where $\Psi_x ( \Lambda , [Z_1 : \dots : Z_N ] )$ stands for $\Psi_x ( \Lambda , u )$ with the identification given by $\prj^{N-1} = SU(N)/ S(U(1) \times U(N-1))$ as above. By recalling the formula (\ref{eqvffspn}) for the Fubini--Study volume form on $\cx^{N-1} \subset \prj^{N-1}$ and writing the above integral in terms of polar coordinates, we find that (\ref{eqcmmmpij}) is zero if $i \neq j$ because of the periodicity of the angle coordinates, by performing the computation as in \cite[Lemma 2.7]{yhhilb} (and as pointed out in Examples \ref{expnlebd} and \ref{expnlebdtp}), since $\Psi_x ( \Lambda , [Z_1 : \dots : Z_N ] )$ does not depend on the angle coordinates as we can see from the formula (\ref{defpsi}).

Thus we find
\begin{equation*}
	\int_{SU(N)} \mu_x (\Lambda  u ) \Psi_x (\Lambda ,  u ) d \sigma_{SU} (u) = \mathrm{diag} (\alpha_1 (\Lambda ) , \dots , \alpha_N (\Lambda ) )
\end{equation*}
for some maps $\alpha_i : \rl^N \to \rl_{\ge 0}$ ($i=1 , \dots , N$); observe that each $\alpha_i$ depends smoothly on $\Lambda$ and is bounded over $\rl^N$, since the $(i,i)$-th entry of the integrand is
\begin{equation*}
	\frac{\Lambda^2_i |(u  \tilde{x})_i|^2 }{\sum_{j=1}^N \Lambda_j^2 |(u \tilde{x})_j|^2} \Psi_{u \tilde{x}} (\Lambda ,  e ) ,
\end{equation*}
where $\tilde{x} \in \cx^N$ is any nonzero lift (i.e.~the homogeneous coordinates) of $x \in \iota (X) \subset \prj^{N-1}$. We further observe that each $\alpha_i$ does not depend on $x \in \iota (X)$, since for any $x' \in \iota (X)$ there exists $u' \in SU(N)$ such that $x' = u'  x$ (as $SU(N)$ acts transitively on the ambient $\prj^{N-1}$) and hence the dependence on $x$ is integrated out by the group invariance of the Haar measure. Moreover, the above formula and (\ref{defpsi}) imply that each $\alpha_i$ can be naturally regarded as a function of $\tilde{\Lambda} = \Lambda^2$, and hence by abuse of notation we shall write $\alpha_i ( \tilde{\Lambda})$ for $\alpha_i ( \Lambda )$, which can be considered as a smooth bounded function on $\rl_{>0}^N$.

 Let $C_N := \{ \eta_1 , \dots , \eta_N \}$ be the group of cyclic permutations of $N$ letters, which is naturally a subgroup of $U(N)$. We then find
\begin{equation*}
	\sum_{i=1}^N \eta_i  \left( \int_{SU(N)} \mu_x (\Lambda  u ) \Psi_x (\Lambda ,  u ) d \sigma_{SU} (u) \right)  \eta^{*}_i = \alpha ( \tilde{\Lambda} ) \cdot \mathrm{id}_{N \times N},
\end{equation*}
with $\alpha ( \tilde{\Lambda} ) := \sum_{i=1}^N \alpha_i ( \tilde{\Lambda})$; we also note that in the above we may assume $\eta_i \in SU(N)$ for $i=1 , \dots , N$ by dividing them by an $N$-th root of $\det ( \eta_i ) \in U(1)$ which leaves the above integral invariant. Thus we get, again by the group invariance of the Haar measure,
\begin{align*}
	&\int_{\eta \in SU(N)} d \sigma_{SU} (\eta ) \int_{u \in SU(N)} \eta \mu_x (\Lambda  u ) \eta^{*} \Psi_x (\Lambda ,  u ) d \sigma_{SU} (u) \\
	&= \frac{1}{N} \sum_{i=1}^N \int_{\eta \eta_i^{-1} \in SU(N)} ( \eta  \eta_i^{-1}) \eta_i \\
	&\quad \quad \quad \quad \quad \quad \left( \int_{u \in SU(N)} \mu_x (\Lambda  u ) \Psi_x (\Lambda ,  u ) d \sigma_{SU} (u) \right)  \eta^{*}_i ( \eta  \eta_i^{-1})^*  d \sigma_{SU} (\eta \eta_i^{-1}) \\
	&=\int_{\eta \in SU(N)} \eta  \left( \frac{1}{N} \sum_{i=1}^N  \eta_i \left( \int_{u \in SU(N)} \mu_x (\Lambda  u ) \Psi_x (\Lambda ,  u ) d \sigma_{SU} (u) \right) \eta^{*}_i \right)  \eta^{*}  d \sigma_{SU} (\eta ) \\
	&= \frac{\alpha ( \tilde{\Lambda} )}{N} \cdot \mathrm{id}_{N \times N} ,
\end{align*}
and hence, by recalling that $\alpha ( \tilde{\Lambda} )$ does not depend on $x \in X$ as pointed out in the above, we find
\begin{align*}
	\mathbb{E}[\bar{\mu}_X (g)] &= \frac{1}{\mathrm{Vol} (\CB )} \int_{x\in \iota(X)} d \nu_{H_e} \int_{ \tilde{\Lambda} \in \rl^{N-1}_{>0} } \Delta^2 (\tilde{\Lambda} ) \gamma (\tilde{\Lambda} ) \rho (\tilde{\Lambda})  d \tilde{\Lambda} \\
	&\quad \quad \quad \quad \int_{\eta \in SU(N)} d \sigma_{SU} (\eta ) \int_{u \in SU(N)} \mu_x (\eta \Lambda  u ) \Psi_x (\Lambda ,  u ) d \sigma_{SU} (u) \\
	&= \left( \frac{\mathrm{Vol} (X , L )}{N \mathrm{Vol} (\CB )}  \int_{\tilde{\Lambda} \in \rl^{N-1}_{>0} } \alpha (\tilde{\Lambda} ) \Delta^2 (\tilde{\Lambda}) \gamma (\tilde{\Lambda}  ) \rho (\tilde{\Lambda}) d \tilde{\Lambda} \right) \cdot \mathrm{id}_{N \times N} .
\end{align*}
Since $\alpha (\tilde{\Lambda})$ is bounded over $\rl_{>0}^N$, the integral
\begin{equation*}
	\int_{\tilde{\Lambda} \in \rl^{N-1}_{>0} } \alpha (\tilde{\Lambda} ) \Delta^2 (\tilde{\Lambda} ) \gamma (\tilde{\Lambda} ) \rho (\tilde{\Lambda}) d \tilde{\Lambda}
\end{equation*}
is a well-defined real number by (\ref{eqrhfvpd}) because $d \sigma_B$ is of finite volume. In fact, the above integral is equal to $\mathrm{Vol} (\CB )$, by observing
\begin{equation*}
	\mathrm{tr} ( \mathbb{E}[\bar{\mu}_X (g)] ) =  \mathrm{Vol} (X , L )
\end{equation*}
since $\mathrm{tr} (\bar{\mu}_X (g)) = \mathrm{Vol} (X , L )$ for all $g \in SL (N , \cx)$ and $d \sigma$ is assumed to have unit volume as in (\ref{equvlsg}). Thus we finally get
\begin{equation*}
	\mathbb{E}[\bar{\mu}_X (g)] = \frac{\mathrm{Vol} (X , L )}{N} \cdot \mathrm{id}_{N \times N}
\end{equation*}
as claimed, with respect to the fixed reference basis $\{ Z_i \}_{i=1}^N$ (see Remark \ref{rmrfbs}). This completes the proof of Theorem \ref{mainthm}.

\begin{remark} \label{rmthmvar}
It is straightforward to extend the above proof to the variant $\bar{\mu}_{X , \nu}$ in Definition \ref{defcomnu}. First note that we have
\begin{equation*}
	\mathbb{E} [ \bar{\mu}_{X , \nu} (g) ] =  \int_{SL(N , \cx )} d \sigma (g) \int_{x\in \iota(X)} \mu_x (g ) d \nu   ,
\end{equation*}
by definition. Noting that $d \nu$ is fixed and does not depend on $g \in SL(N , \cx )$, we again apply Fubini's theorem to $SL(N , \cx ) \times \iota (X)$ with the product measure $d \sigma \times d \nu $, to find
\begin{equation} \label{eqavcomnu}
	\mathbb{E} [ \bar{\mu}_{X , \nu} (g) ] =   \int_{x\in \iota(X)} d \nu \int_{SL(N , \cx )}  \mu_x (g ) d \sigma (g) 
\end{equation}
and repeat the argument presented above.
%Noting that this is essentially the same as (\ref{eqavcomacf}), we can repeat the above proof by replacing (\ref{eqavcomacf}) with (\ref{eqavcomnu}) and find that exactly the same argument applies almost word by word.
\end{remark}

\addcontentsline{toc}{section}{References}
\bibliography{random.bib}

\begin{flushleft}
{\footnotesize
Department of Mathematics, Tokyo Institute of Technology, \\
2-12-1 Ookayama, Meguro-ku, Tokyo, 152-8551, Japan. \\
Email: \texttt{hashimoto@math.titech.ac.jp}}
\end{flushleft}

\begin{flushleft}
{\footnotesize
Current address: Department of Mathematics, Osaka Metropolitan University, \\
3-3-138, Sugimoto, Sumiyoshi-ku, Osaka, 558-8585, Japan. \\
Email: \texttt{yhashimoto@omu.ac.jp}}
\end{flushleft}

\end{document}